\documentclass[11pt]{article}
\usepackage{amsthm, amsmath, amssymb, amsfonts, url, booktabs, tikz, setspace, fancyhdr, bm}
\usepackage{geometry}
\usepackage{hyperref, enumerate}
\usepackage[shortlabels]{enumitem}
\usepackage[babel]{microtype}
\usepackage[english]{babel}
\usepackage[capitalise]{cleveref}
\usepackage{comment}
\usepackage{bbm}
\usepackage{csquotes}
\usepackage{mathabx}
\usepackage{tikz}
\usepackage{graphicx}
\usepackage{float}
\usepackage{xcolor}
\usepackage{mathtools}
\usepackage{mathrsfs}
\usetikzlibrary{positioning, arrows.meta, shapes.geometric}

\counterwithin{figure}{section}

\newtheorem{theorem}{Theorem}[section]
\newtheorem{prop}[theorem]{Proposition}
\newtheorem{lemma}[theorem]{Lemma}
\newtheorem{cor}[theorem]{Corollary}

\newtheorem{claim}[theorem]{Claim}

\usetikzlibrary{decorations.pathmorphing}

\theoremstyle{definition}
\newtheorem{defn}[theorem]{Definition}
\newtheorem*{defn-non}{Definition}

\newlist{Case}{enumerate}{2}
\setlist[Case, 1]{%
    label           =   {\bfseries Case \arabic*.},
    labelindent=1em ,labelwidth=1.3cm, labelsep*=1em, leftmargin =!
}
\setlist[Case, 2]{%
    label           =   {\bfseries Subcase \arabic{Casei}.\arabic*.},
    labelindent=-1em ,labelwidth=1.3cm, labelsep*=1em, leftmargin =!
}

\newenvironment{poc}{\begin{proof}[Proof of the claim]}{\end{proof}}

\usepackage{todonotes}

\title{Improvement on the Erd\H{o}s-Kleitman conjecture via the KKL theorem} 
\author{
Gennian Ge\thanks{School of Mathematical Sciences, Capital Normal University, Beijing, China. Email: gnge@zju.edu.cn. Gennian Ge is supported by the National Key Research and Development Program of China under Grant 2025YFC3409900, the National Natural Science Foundation of China under Grant 12231014, and Beijing Scholars Program.}
\and 
Jialuo Wang\thanks{School of Mathematical Sciences, University of Science and Technology of China,
Hefei, China. Email: 673391676@qq.com}
\and
Zixiang Xu\thanks{Extremal Combinatorics and Probability Group (ECOPRO), Institute for Basic Science (IBS), Daejeon, South Korea. Email: zxxu8023@qq.com. Supported by IBS-R029-C4.}
}
\date{}
\begin{document}
\maketitle
\begin{abstract}
In 1974, Erd\H{o}s and Kleitman conjectured that if a family \(\mathcal{F}\subseteq 2^{[n]}\) contains no matching of size \(s\) and is maximal with respect to this property, then
\(
|\mathcal{F}|\ge \left(1-2^{-(s-1)}\right)\cdot 2^{n}.
\)
For decades, the best general lower bound remained the trivial \(2^{n-1}\). About a decade ago, Frankl and Tokushige emphasized that obtaining a bound of the form \(\left(\frac{1}{2}+\varepsilon\right)\cdot 2^n\) for some \(\varepsilon>0\) is a challenging problem. A breakthrough of Buci\v{c}, Letzter, Sudakov and Tran in 2018 showed that
\(
|\mathcal{F}|\ge \left(1-\frac{1}{s}\right)\cdot 2^n
\)
via two very elegant and quite different approaches. Our main result shows that \[
|\mathcal{F}|\ge \left(
      1 - \frac{1}{s + (s-2)\frac{\log n}{2\sqrt{5}n}}
   \right)\cdot 2^n 
\] by exploiting a connection to the cornerstone result of Kahn, Kalai and Linial on influences of Boolean functions.
Independently, we can also obtain a weaker improvement combining the linear algebra method with a combinatorial twist.

\end{abstract}

\section{Introduction}
Let $[n]:=\{1,2,\dots,n\}$. A \emph{matching} of size $s$ contains $s$ pairwise disjoint sets. For given positive integer \(s\ge 2\), we call a family $\mathcal{F}\subseteq 2^{[n]}$ \(s\)-saturated if it contains no \(s\) pairwise disjoint sets, and moreover no set can be added to $\mathcal{F}$ while preserving this property. An old conjecture of Erd\H{o}s and Kleitman~\cite{1974ErdosKleitman} asserts that every $s$-saturated family has size at least
\[
\bigl(1-2^{-(s-1)}\bigr)\cdot 2^{n}.
\]
For decades, the best general lower bound remained the trivial \(2^{n-1}\), which follows from results in~\cite{1984ErdosDM,1961EKR} implying that any maximal matching-free family occupies at least half of the discrete cube. In particular, Frankl and Tokushige pointed out in~\cite{2016JCTAFrankl} that obtaining a bound of the form \((\tfrac{1}{2}+\varepsilon)\cdot 2^n\) for some \(\varepsilon>0\) already appears to be a difficult problem. A breakthrough was achieved in 2018 by Buci\v{c}, Letzter, Sudakov and Tran~\cite{2018BLMS}, who established such an improvement and went significantly beyond the trivial bound.
\begin{theorem}[\cite{2018BLMS}]\label{thm:2018BLMS}
    Let \(s\ge 2\) be an integer, and let \(\mathcal{F}\subseteq 2^{[n]}\) be \(s\)-saturated. Then
    \[
        |\mathcal{F}|\ge \Bigl(1-\frac{1}{s}\Bigr)\cdot 2^{n}.
    \]
\end{theorem}
In~\cite{2018BLMS}, the authors obtained such an improvement by means of two conceptually different proofs: one exploits an intriguing connection to correlation inequalities, while the other is algebraic in nature. In particular, they remarked in~\cite{2018BLMS} that a further improvement of their lower bound should be possible by running the argument of their proof more carefully. Our main result strengthens their bound by combining this correlation-inequality approach with the celebrated Kahn-Kalai-Linial influence inequality for Boolean functions~\cite{1988kahanKalaiLinial}.

\begin{theorem}\label{thm:ViaKKL}
    Let \(s\ge 2\) be an integer, and let \(\mathcal{F}\subseteq 2^{[n]}\) be an \(s\)-saturated family. Then
  \[
   |\mathcal{F}|
   \;\ge\;
   \left(
      1 - \frac{1}{\,s + (s-2)\frac{\log n}{2\sqrt{5}n}\,}
   \right)\cdot 2^n.
\]
\end{theorem}

In addition, we refine the algebraic method used in~\cite{2018BLMS} by introducing a new combinatorial twist that yields an additional exponential contribution. For integers \(n\ge 1\) and \(2\le s\le n+1\), define \[
g(n,s):=\max\left\{t\in\{0,1,\dots,n\}:
\sum_{i=0}^{t}\binom{n}{i}
\le
\frac{2^{n}-(2^{s-1}-s)\,2^{t}}{s}
\right\}.
\]
\begin{theorem}\label{thm:main}
    For positive integer \(s\ge 2\), let \(\mathcal{F}\subseteq 2^{[n]}\) be an \(s\)-saturated family. Then 
    \[|\mathcal{F}|\ge \bigg(1-\frac{1}{s}\bigg)\cdot 2^{n} + \bigg(\frac{2^{s-1}}{s}-1\bigg)\cdot 2^{g(n,s)}. \]
\end{theorem}
We remark that the bound in~\cref{thm:main} is weaker than that of~\cref{thm:ViaKKL}, since asymptotically \(g(n,s)=\frac{n}{2}-c_s\sqrt{n}\) for some constant \(c_s>0\) depending only on \(s\). Nevertheless, we include the full proof, as the underlying idea may be of independent interest. In particular, the strategy of starting from a polynomial method and then enhancing it by exploiting additional combinatorial structure is not specific to the present setting, but it is also far from routine. Similar combinations of algebraic and combinatorial ingredients appear, for instance, in work on the \(k\)-neighborly box problem~\cite{1997NogaAlon,2013DM,2012HuangSudakov}, in stability results for Kleitman's theorem~\cite{2023GaoLiuXu}, in Alon-Babai-Suzuki type results~\cite{2026AiLiuABS,1991JCTAABS,2015EUJCYounjin,2018WangWeiGe} and in the study of the \(s\)-distance problem~\cite{1981Distance1,1983Distance2,2021PAMSdistance}.

Moreover, \cref{thm:2018BLMS} is derived in~\cite{2018BLMS} as a corollary of a stronger multipartite statement. Following their terminology, a sequence of families \(\mathcal{F}_1,\ldots,\mathcal{F}_s \subseteq 2^{[n]}\) is called \emph{cross \(s\)-saturated} if there are no \(s\) pairwise disjoint sets \(A_i \in \mathcal{F}_i\) for \(i\in[s]\), and the sequence is maximal with respect to this property. Buci\v{c}, Letzter, Sudakov and Tran~\cite[Theorem~5]{2018BLMS} proved that every cross \(s\)-saturated sequence satisfies
\[
|\mathcal{F}_1|+\cdots+|\mathcal{F}_s| \;\ge\; (s-1)\,2^n,
\]
which implies \cref{thm:2018BLMS} and is best possible, for example by taking \(\mathcal{F}_1=\emptyset\) and \(\mathcal{F}_2=\cdots=\mathcal{F}_s=2^{[n]}\). While this multipartite bound is sharp, the lower bound in \cref{thm:2018BLMS} for a single \(s\)-saturated family can always be improved by~\cref{thm:ViaKKL}. We view this discrepancy between the tight multipartite theorem and its non-tight single-family corollary as a rather interesting phenomenon, which has also appeared in other problems, such as the Erd\H{o}s-Ginzburg-Ziv problem~\cite{2021IsrLisa,sauermann2023erdhosginzburgzivproblemlargedimension}.
\section{Notation and useful tools}
For a family $\mathcal{A}\subseteq 2^{V}$, we write
\[
   \overline{\mathcal{A}}
   \coloneqq
   \{\, V\setminus A : A\in\mathcal{A}\,\}
   \quad\text{and}\quad
   \mathcal{A}^c
   \coloneqq
   2^{V}\setminus\mathcal{A}.
\]
In particular, 
\[
   \overline{\mathcal{A}^c}
   = (\overline{\mathcal{A}})^c
   = \{\, V\setminus A : A\notin\mathcal{A}\,\}.
\]
We say that a family \(\mathcal{F}\subseteq 2^{V}\) is \emph{increasing} if \(F\in\mathcal{F}\) implies that every superset \(F'\supseteq F\) also belongs to \(\mathcal{F}\). We now recall the notion of disjoint occurrence from~\cite{2018BLMS}.
\begin{defn}
Let $\mathcal{A},\mathcal{B}\subseteq 2^{[n]}$ be increasing families.  
The \emph{disjoint occurrence} of $\mathcal{A}$ and $\mathcal{B}$ is defined as
\[
  \mathcal{A}\;\Box\;\mathcal{B}
  \coloneqq 
  \{\,A\sqcup B : A\in\mathcal{A},\ B\in\mathcal{B},\ A\cap B=\emptyset\,\}.
\]
More generally, for a family $\mathcal{A}\subseteq 2^{[n]}$ and an integer $m\ge 1$, we denote by
$\mathcal{A}^{\Box m}$ the disjoint occurrence of $m$ copies of $\mathcal{A}$.  
In particular, we set $\mathcal{A}^{\Box 0}\coloneqq 2^{[n]}$.
\end{defn}
The following inequality for the disjoint occurrence plays a crucial role in~\cite{2018BLMS}.
\begin{lemma}[\cite{2018BLMS}]\label{lem:ineq of disjoint occ}
    Let $\mathcal{A},\mathcal{B}\subseteq 2^{[n]}$ be increasing families. Then $|\mathcal{A}\;\Box\; \mathcal{B}|\le |\overline{\mathcal{A}}\cap \mathcal{B}|$.
\end{lemma}

It was shown in~\cite{2018BLMS} that any \(s\)-saturated family is increasing. For completeness, we include the short proof.
\begin{prop}[\cite{2018BLMS}]\label{claim:Increasing}
Let \(\mathcal{F}\subseteq 2^{[n]}\) be an \(s\)-saturated family, then \(\mathcal{F}\) is increasing.
\end{prop}
\begin{proof}[Proof of Proposition~\ref{claim:Increasing}]
Let $F\in \mathcal{F}$. Suppose there is some $F'\notin \mathcal{F}$ such that \(F\subseteq F'\). Since adding \(F'\) to \(\mathcal{F}\) can induce a matching of size \(s\), then there exist $s-1$ pairwise disjoint sets $F_1,\ldots ,F_{s-1}\in\mathcal{F}$ such that $F'\cap\big(\bigsqcup_{i=1}^{s-1}F_i\big)=\emptyset$. Since \(F\subseteq F'\), we can see $F,F_1,\ldots,F_{s-1}$ form a matching of size \(s\), a contradiction.
\end{proof}
The next structural relation between \(\mathcal{F}\) and its complement, established in~\cite{2018BLMS}, will be of central importance for our argument.
\begin{prop}[\cite{2018BLMS}]\label{claim:RelationFG}
    Let \(\mathcal{F}\subseteq 2^{[n]}\) be an \(s\)-saturated family, then \(
    \overline{\mathcal{F}^{c}} = \mathcal{F}^{\Box (s-1)}.
    \)
\end{prop}
We then briefly recall the notion of influence from the analysis of Boolean functions. Let \(f:\{0,1\}^n\to\{0,1\}\) be a Boolean function and let \(S\subseteq [n]\).  
We define \(I_f(S)\) as follows: choose the coordinates \((x_i)_{i\notin S}\) independently and uniformly at random from \(\{0,1\}\), and consider the resulting restricted function in the remaining variables \((x_i)_{i\in S}\). We say that \(f\) is \emph{undetermined} by this partial assignment if there exist two choices of \((x_i)_{i\in S}\) for which \(f\) takes different values, and we set \(I_f(S)\) to be the probability over the random choice of \((x_i)_{i\notin S}\) that \(f\) remains undetermined.

For a single coordinate \(i\in[n]\), we abbreviate \(I_f(\{i\})\) by \(I_f(x_i)\). Equivalently, \(I_f(x_i)\) is the probability that, after randomly fixing all \(x_j\) with \(j\neq i\), changing the value of \(x_i\) flips the value of \(f\).

We shall use the following form of the celebrated Kahn--Kalai--Linial theorem~\cite{1988kahanKalaiLinial}.

\begin{theorem}[\cite{1988kahanKalaiLinial}]\label{lemma:Bool}
   Let \(f:\{0,1\}^n\to\{0,1\}\) take the value \(1\) with probability \(p\) with respect to the uniform measure on \(\{0,1\}^n\), and assume \(p\le \tfrac{1}{2}\). Then
   \[
      \sum_{i=1}^n \bigl(I_f(x_i)\bigr)^2
      \;\ge\;
      \frac{p^2\log^2 n}{5n}.
   \]
   In particular, there exists some \(i\in[n]\) such that
   \(
      I_f(x_i)\;\ge\; \frac{p\log n}{\sqrt{5}n}.
   \)
\end{theorem}

We now translate \cref{lemma:Bool} into a convenient formulation for set systems.  
Given \(\mathcal{A}\subseteq 2^{[n]}\), define the associated Boolean function
\(f_{\mathcal{A}}:\{0,1\}^n\to\{0,1\}\) by letting \(\bm{x}\in\{0,1\}^n\) be the characteristic vector of a set \(S\subseteq[n]\) and setting
\[
   f_{\mathcal{A}}(\bm{x})
   \;=\;
   \begin{cases}
      1, & \text{if } S\in\mathcal{A},\\[2pt]
      0, & \text{otherwise}.
   \end{cases}
\]
Then \(f_{\mathcal{A}}=1\) with probability
\(
   p = |\mathcal{A}|/2^n.
\)

Fix \(i\in[n]\). After randomly fixing all coordinates \(x_j\) with \(j\neq i\), the remaining choice of \(x_i\in\{0,1\}\) corresponds precisely to the pair \((A, A\cup\{i\})\), where
\[
   A = \{j\in[n]\setminus\{i\} : x_j = 1\}.
\]
The function \(f_{\mathcal{A}}\) is undetermined after this partial assignment if and only if exactly one of \(A\) and \(A\cup\{i\}\) belongs to \(\mathcal{A}\). Each choice of \((x_j)_{j\neq i}\) is equally likely and corresponds to a unique subset \(A\subseteq [n]\setminus\{i\}\), so
\[
   I_{f_{\mathcal{A}}}(x_i)
   \;=\;
   \frac{1}{2^{n-1}}
   \bigl|\{A\subseteq[n]\setminus\{i\} :
        |\{A, A\cup\{i\}\}\cap\mathcal{A}|=1\}\bigr|.
\]

When \(|\mathcal{A}|\le 2^{n-1}\), applying \cref{lemma:Bool} to \(f_{\mathcal{A}}\) with \(p = |\mathcal{A}|/2^n\le\frac{1}{2}\), we obtain an index \(i\in[n]\) such that
\[
   I_{f_{\mathcal{A}}}(x_i)
   \;\ge\;
   \frac{p\log n}{\sqrt{5}n}
   =
   \frac{\log n}{\sqrt{5}n}\cdot\frac{|\mathcal{A}|}{2^n}.
\]
Combining this with the expression for \(I_{f_{\mathcal{A}}}(x_i)\) above and multiplying both sides by \(2^{n-1}\), we arrive at the following consequence.

\begin{lemma}\label{lem:KKL}
    For any
    \(\mathcal{A}\subseteq 2^{[n]}\) with \(|\mathcal{A}|\le 2^{n-1}\),
    there exists some \(i\in [n]\) such that
    \[
        \bigl|\{A\subseteq [n]\setminus\{i\}:
               |\{A, A\cup\{i\}\}\cap\mathcal{A}|=1\}\bigr|
        \;\ge\; \frac{\log n}{2\sqrt{5}n}\,|\mathcal{A}|.
    \]
\end{lemma}

\section{Proof of~\cref{thm:ViaKKL}}
Let $\mathcal{F}\subseteq 2^{[n]}$ be an $s$-saturated family. Notice that by Proposition~\ref{claim:RelationFG}, we have
\[
  \overline{\mathcal{F}^{c}} = \mathcal{F}^{\Box (s-1)}.
\]
In this view, our task reduces to obtaining
a suitable upper bound on $\bigl|\mathcal{F}^{\Box (s-1)}\bigr|$.
Fix an arbitrary element $x\in [n]$, and consider the two families of subsets of $[n]\setminus\{x\}$ defined by
\[
    \mathcal{F}(x)
    \coloneqq
    \{\,F\setminus\{x\}: F\in\mathcal{F},\, x\in F\,\},
    \qquad
    \mathcal{F}(\overline{x})
    \coloneqq
    \{\,F\in\mathcal{F}: x\notin F\,\}.
\]
The following containment is the key observation in our proof.

\begin{claim}\label{prop:inclusion}
 For any integer $t$ with $0\le t\le s-1$, the set systems in \(2^{[n]\setminus\{x\}}\) satisfy
   \[
      \overline{\bigl(\mathcal{F}(x)\;\Box\; \mathcal{F}(\overline{x})^{\Box t}\bigr)^{c}}
      \supseteq
      \mathcal{F}(\overline{x})^{\Box (s-1-t)}.
   \]
\end{claim}
\begin{poc}
Let $U\coloneqq [n]\setminus\{x\}$.
Take any
\[
   A_1\sqcup \dots \sqcup A_{s-1-t}\in \mathcal{F}(\overline{x})^{\Box (s-1-t)},
\]
where $A_1,\dots,A_{s-1-t}$ are pairwise disjoint sets in $\mathcal{F}(\overline{x})$.
Set
\[
   B \coloneqq A_1\sqcup \dots \sqcup A_{s-1-t},
   \qquad
   C \coloneqq U\setminus B.
\]
To prove the claim, it suffices to show that $C\notin \mathcal{F}(x)\;\Box\; \mathcal{F}(\overline{x})^{\Box t}$.

Observe that $C\in \mathcal{F}(x)\;\Box\; \mathcal{F}(\overline{x})^{\Box t}$ would imply
\[
   [n]\setminus B \in \mathcal{F}^{\Box (t+1)}.
\]
Indeed, if
\[
   C = C_0 \sqcup C_1 \sqcup \dots \sqcup C_t
\]
with $C_0\in\mathcal{F}(x)$ and $C_i\in\mathcal{F}(\overline{x})$ for $1\le i\le t$, then by definition
there exist $F_0\in\mathcal{F}$ with $x\in F_0$ and $F_0\setminus\{x\}=C_0$, and
$F_i = C_i\in\mathcal{F}$ with $x\notin F_i$ for $1\le i\le t$.
These $t+1$ sets $F_0,\dots,F_t$ are pairwise disjoint and moreover \([n]\setminus B\) can be expressed as
\[
  [n]\setminus B = C\cup\{x\} = (C_0\cup\{x\}) \sqcup C_1 \sqcup\dots\sqcup C_t
   = F_0\sqcup \dots \sqcup F_t,
\]
which yields that $[n]\setminus B\in\mathcal{F}^{\Box (t+1)}$.

Therefore, to show $C\notin \mathcal{F}(x)\;\Box\; \mathcal{F}(\overline{x})^{\Box t}$, it is enough to show that
\[
   [n]\setminus B \notin \mathcal{F}^{\Box (t+1)}.
\]
Suppose for a contradiction that
\(
   [n]\setminus B \in \mathcal{F}^{\Box (t+1)}.
\)
Then there exist $t+1$ pairwise disjoint sets, namely
\(
   A_{s-t},\dots,A_s \in \mathcal{F}
\)
such that
\[
   [n]\setminus B
   = A_{s-t}\sqcup\dots\sqcup A_s.
\]
Since $B = A_1\sqcup \dots \sqcup A_{s-1-t}$ and $[n]\setminus B$ are complementary, the $s$ sets
\[
   A_1,\dots,A_{s-1-t},A_{s-t},\dots,A_s
\]
are pairwise disjoint members of $\mathcal{F}$, forming a matching of size $s$ in $\mathcal{F}$, a contradiction. This finishes the proof.
\end{poc}
Notice that every member of $\mathcal{F}(\overline{x})^{\Box (s-1)}$ is a union of
$s-1$ pairwise disjoint sets from $\mathcal{F}$ that all avoid $x$, and conversely
every element of $\mathcal{F}^{\Box (s-1)}$ avoiding $x$ arises in this way. Hence
\begin{equation}\label{eq:exchange1}
    \mathcal{F}(\overline{x})^{\Box (s-1)}
   = \mathcal{F}^{\Box (s-1)}(\overline{x}).
\end{equation}
Then the following inequality reduces the problem of bounding $|\mathcal{F}^{\Box (s-1)}|$
to bounding a suitable section of $\mathcal{F}$. In particular, it provides an
upper bound on $|\mathcal{F}^{\Box (s-1)}|$ in terms of
$|\mathcal{F}(\overline{x})^{\Box (s-1)}|$.

\begin{claim}\label{prop:bound diversity}
 For every $x\in[n]$, we have
\[
   (s-1)\cdot\bigl|\mathcal{F}^{\Box (s-1)}\bigr|
   \le
   (s-2)\cdot\bigl|\mathcal{F}(\overline{x})^{\Box (s-1)}\bigr|
   + 2^{n-1}.
\]
\end{claim}
\begin{poc}
Fix $x\in[n]$ and $t\in\{0,\dots,s-2\}$. Using the decomposition according to whether a set contains $x$, we first write
\begin{align*}
   \bigl|\mathcal{F}^{\Box(s-1)}\bigr|
   &= \bigl|\mathcal{F}^{\Box(s-1)}(\overline{x})\bigr|
      + \bigl|\mathcal{F}^{\Box(s-1)}(x)\bigr|\\
   &= \bigl|\mathcal{F}(\overline{x})^{\Box(s-1)}\bigr|
      + \bigl|\mathcal{F}(x)\;\Box\;\mathcal{F}(\overline{x})^{\Box(s-2)}\bigr|,
\end{align*}
where we take advantage of~\eqref{eq:exchange1}.

Next we apply Lemma~\ref{lem:ineq of disjoint occ} with
\[
   \mathcal{A} = \mathcal{F}(x)\;\Box\;\mathcal{F}(\overline{x})^{\Box t}
   \quad\text{and}\quad
   \mathcal{B} = \mathcal{F}(\overline{x})^{\Box(s-2-t)}.
\]
Notice that
\[
   \mathcal{A}\;\Box\;\mathcal{B}
   = \bigl(\mathcal{F}(x)\;\Box\;\mathcal{F}(\overline{x})^{\Box t}\bigr)
     \;\Box\;\mathcal{F}(\overline{x})^{\Box(s-2-t)}
   = \mathcal{F}(x)\;\Box\;\mathcal{F}(\overline{x})^{\Box(s-2)},
\]
so by Lemma~\ref{lem:ineq of disjoint occ},
\begin{align*}
   \bigl|\mathcal{F}^{\Box(s-1)}\bigr|
   &\le \bigl|\mathcal{F}(\overline{x})^{\Box(s-1)}\bigr|
      + \bigl|\mathcal{F}(x)\Box\mathcal{F}(\overline{x})^{\Box(s-2)}\bigr|\\
   &\le \bigl|\mathcal{F}(\overline{x})^{\Box(s-1)}\bigr|
      + \bigl|\overline{\mathcal{A}}\cap\mathcal{B}\bigr|.
\end{align*}
For any families $\mathcal{A},\mathcal{B}$ we have
\(
   \bigl|\overline{\mathcal{A}}\cap\mathcal{B}\bigr|
   = |\mathcal{B}| - \bigl|\overline{\mathcal{A}^{c}}\cap\mathcal{B}\bigr|.
\)
 Therefore
\begin{align*}
   \bigl|\mathcal{F}^{\Box(s-1)}\bigr|
   &\le \bigl|\mathcal{F}(\overline{x})^{\Box(s-1)}\bigr|
      + \bigl|\mathcal{F}(\overline{x})^{\Box(s-2-t)}\bigr|
      - \bigl|\overline{\mathcal{A}^c}
               \cap\mathcal{F}(\overline{x})^{\Box(s-2-t)}\bigr|.
\end{align*}
By Claim~\ref{prop:inclusion},
\[
   \overline{\bigl(\mathcal{F}(x)\;\Box\;\mathcal{F}(\overline{x})^{\Box t}\bigr)^{c}}
   \supseteq \mathcal{F}(\overline{x})^{\Box(s-1-t)},
\]
and clearly
\[
   \mathcal{F}(\overline{x})^{\Box(s-2-t)}
   \supseteq \mathcal{F}(\overline{x})^{\Box(s-1-t)}.
\]
Hence
\[
   \bigl|\overline{\mathcal{A}^c}
          \cap\mathcal{F}(\overline{x})^{\Box(s-2-t)}\bigr|
   \ge
   \bigl|\mathcal{F}(\overline{x})^{\Box(s-1-t)}\bigr|,
\]
and we obtain
\begin{equation}\label{eq:bound F(overline x)}
   \bigl|\mathcal{F}^{\Box(s-1)}\bigr|
   \le \bigl|\mathcal{F}(\overline{x})^{\Box(s-1)}\bigr|
      + \bigl|\mathcal{F}(\overline{x})^{\Box(s-2-t)}\bigr|
      - \bigl|\mathcal{F}(\overline{x})^{\Box(s-1-t)}\bigr|.
\end{equation}

Now let \(t\) run from \(0\) to \(s-2\) and sum the inequalities in \eqref{eq:bound F(overline x)}. The left-hand side gives
\[
   (s-1)\,\bigl|\mathcal{F}^{\Box(s-1)}\bigr|.
\]
On the right-hand side, the first term contributes
\[
   (s-1)\,\bigl|\mathcal{F}(\overline{x})^{\Box(s-1)}\bigr|,
\]
while
\[
   \sum_{t=0}^{s-2} \bigl|\mathcal{F}(\overline{x})^{\Box(s-2-t)}\bigr|
   = \sum_{j=0}^{s-2} \bigl|\mathcal{F}(\overline{x})^{\Box j}\bigr|,
\]
and
\[
   \sum_{t=0}^{s-2} \bigl|\mathcal{F}(\overline{x})^{\Box(s-1-t)}\bigr|
   = \sum_{k=1}^{s-1} \bigl|\mathcal{F}(\overline{x})^{\Box k}\bigr|.
\]
Their difference is
\[
   \sum_{j=0}^{s-2} \bigl|\mathcal{F}(\overline{x})^{\Box j}\bigr|
   - \sum_{k=1}^{s-1} \bigl|\mathcal{F}(\overline{x})^{\Box k}\bigr|
   = \bigl|\mathcal{F}(\overline{x})^{\Box 0}\bigr|
     - \bigl|\mathcal{F}(\overline{x})^{\Box(s-1)}\bigr|.
\]
Since \(\mathcal{F}(\overline{x})^{\Box 0} = 2^{[n]\setminus\{x\}}\), we have
\(\bigl|\mathcal{F}(\overline{x})^{\Box 0}\bigr| = 2^{n-1}\). Putting everything together,
\[
   (s-1)\,\bigl|\mathcal{F}^{\Box(s-1)}\bigr|
   \le (s-1)\,\bigl|\mathcal{F}(\overline{x})^{\Box(s-1)}\bigr|
      + 2^{n-1}
      - \bigl|\mathcal{F}(\overline{x})^{\Box(s-1)}\bigr|,
\]
which simplifies to
\[
   (s-1)\,\bigl|\mathcal{F}^{\Box(s-1)}\bigr|
   \le (s-2)\,\bigl|\mathcal{F}(\overline{x})^{\Box(s-1)}\bigr|
      + 2^{n-1}.
\]
This finishes the proof.

\end{poc}

We now need an upper bound on
\(
   \min_{x\in[n]} \bigl|\mathcal{F}^{\Box(s-1)}(\overline{x})\bigr|.
\)
Note that if $\mathcal{A}\subseteq 2^{[n]}$ is increasing, then
$\mathcal{A}(i)\supseteq\mathcal{A}(\overline{i})$, and each pair
$\{A,A\cup\{i\}\}$ contributes to the left-hand side of Lemma~\ref{lem:KKL} if and only if exactly one of $A,A\cup\{i\}$ lies in
$\mathcal{A}$. Hence
\[
   \bigl|\{A\subseteq [n]\setminus\{i\}:
          |\{A, A\cup\{i\}\}\cap\mathcal{A}|=1\}\bigr|
   = |\mathcal{A}(i)|-|\mathcal{A}(\overline{i})|
   = |\mathcal{A}|-2|\mathcal{A}(\overline{i})|.
\]
Applying Lemma~\ref{lem:KKL} to an increasing family
$\mathcal{A}\subseteq 2^{[n]}$ with $|\mathcal{A}|\le 2^{n-1}$, we
obtain an index $i\in[n]$ such that
\begin{equation}\label{eq:KKL-section-bound}
   2|\mathcal{A}(\overline{i})|
   \;\le\;
   \Bigl(1-\frac{\log n}{2\sqrt{5}n}\Bigr)\,|\mathcal{A}|.
\end{equation}

Now we apply this to
\(
   \mathcal{A}: = \mathcal{F}^{\Box(s-1)}.
\)
Since $\mathcal{F}$ is increasing, so is
$\mathcal{F}^{\Box(s-1)}$; moreover, as observed above,
$\mathcal{F}^{\Box(s-1)}$ is intersecting, and therefore
\(
   \bigl|\mathcal{F}^{\Box(s-1)}\bigr|\le 2^{n-1}.
\)
Thus, by \eqref{eq:KKL-section-bound}, there exists
$x\in[n]$ such that
\begin{equation}\label{eq:good-x-section}
   \bigl|\mathcal{F}^{\Box(s-1)}(\overline{x})\bigr|
   \;\le\;
   \frac{1}{2}\Bigl(1-\frac{\log n}{2\sqrt{5}n}\Bigr)
   \bigl|\mathcal{F}^{\Box(s-1)}\bigr|.
\end{equation}

Combining Claim~\ref{prop:bound diversity} with \eqref{eq:good-x-section},
we obtain
\[
   (s-1)\,\bigl|\mathcal{F}^{\Box(s-1)}\bigr|
   \;\le\;
   (s-2)\,\bigl|\mathcal{F}^{\Box(s-1)}(\overline{x})\bigr|
   + 2^{n-1}
   \;\le\;
   \frac{s-2}{2}\Bigl(1-\frac{\log n}{2\sqrt{5}n}\Bigr)
   \bigl|\mathcal{F}^{\Box(s-1)}\bigr|
   + 2^{n-1}.
\]
Rearranging, we get
\begin{equation}\label{eq:Fbox-upper}
   \bigl|\mathcal{F}^{\Box(s-1)}\bigr|
   \;\le\;
   \frac{2^n}{\,s + (s-2)\frac{\log n}{2\sqrt{5}n}\,}.
\end{equation}

Recall that
\(
   \overline{\mathcal{F}^{c}} = \mathcal{F}^{\Box(s-1)}
\),
so $|\mathcal{F}| = 2^n - |\mathcal{F}^{\Box(s-1)}|$.  Using
\eqref{eq:Fbox-upper}, we obtain
\[
   |\mathcal{F}|
   \;\ge\;
   2^n\left(
      1 - \frac{1}{\,s + (s-2)\frac{\log n}{2\sqrt{5}n}\,}
   \right),
\]
which is the desired bound. This finishes the proof.

\section{Proof of~\cref{thm:main}}
Assume \(2\le s\le n+1\), since the case \(s>n+1\) is trivial.
Let \(\mathcal{F}\subseteq 2^{[n]}\) be an \(s\)-saturated family, and define
\(
\mathcal{G}:=\overline{\mathcal{F}^c}.
\)
Since \(|\mathcal{F}|=2^n-|\mathcal{G}|\), it suffices to prove that
\begin{equation}\label{eq:G-upper-main}
|\mathcal{G}|
\le
\frac{2^{n}-(2^{s-1}-s)\,2^{g(n,s)}}{s}.
\end{equation}
Notice that Proposition\ref{claim:Increasing} and Proposition~\ref{claim:RelationFG} together give the following corollary.
\begin{cor}\label{Cor:IncrGG}
   \(\mathcal{G}\) is increasing.
\end{cor}

Let \(k\in\{0,1,2,\ldots,s-1\}\) be an integer. Let \(\mathcal{W}_{k}\) be the collection of all possible polynomials of the form
\begin{equation}\label{formW}
    w(\bm{x})=\prod\limits_{j\in\sqcup_{h\in S}A_{h}}x_{j}\cdot \prod\limits_{j\in\sqcup_{h\in T}A_{h}}(1-x_{j}),
\end{equation}
where \(S\) and \(T\) are disjoint, \(S\sqcup T=[s-1]\), \(|S|=k\), and $A_1,\ldots,A_{s-1}$ are pairwise disjoint sets in $\mathcal{F}$.

Consider the vector space of functions from \(\{0,1\}^{n}\) to \(\mathbb{R}\). Notice that this is a vector space of dimension \(2^{n}\) and all of the polynomials in \(\{\mathcal{W}_{k}\}_{k\in\{0,1,\ldots,s-1\}}\) lie in this vector space. For each \(k\in\{0,1,\ldots,s-1\}\), let
\(
\operatorname{span}(\mathcal{W}_{k})
\)
denote the linear span of \(\mathcal{W}_{k}\) in this space.

For each subset \(S\subseteq [n]\), we define the indicator polynomial \(p_{S}:\{0,1\}^{n}\rightarrow \mathbb{R}\) as
\[
p_{S}(\bm{x})=\prod_{i\in S}x_{i},
\]
where \(\bm{x}=(x_{1},\ldots,x_{n})\). We will take advantage of the following result in~\cite{2018BLMS}.
\begin{prop}[Lemma~8,~\cite{2018BLMS}]\label{prop:linearly independent of p_s}
    The polynomials in \(\{p_{S}\}_{S\subseteq [n]}\) are linearly independent in \(\mathbb{R}\).
\end{prop}

By definition and by Proposition~\ref{claim:RelationFG}, for given \(k\in\{0,1,2,\ldots,s-1\}\), each \(G\in\mathcal{G}\) can correspond to at least \(\binom{s-1}{k}\) distinct polynomials in \(\mathcal{W}_{k}\) of the form~\eqref{formW}.

On the other hand, let \(\mathcal{W}_{G}\) be the collection of all possible polynomials corresponding to \(G\), then \(|\mathcal{W}_{k}\cap\mathcal{W}_{G}|\ge\binom{s-1}{k}\) for any \(G\in\mathcal{G}\) and any \(k\in\{0,1,\ldots,s-1\}\), therefore for any given \(G\in\mathcal{G}\), we have
\[
|\mathcal{W}_{G}|\ge \sum\limits_{k=0}^{s-1}\binom{s-1}{k}=2^{s-1}.
\]

We equip the vector space of functions \(f:\{0,1\}^n\to\mathbb R\) with the inner product
\[
\langle f,g\rangle := \sum_{\bm{x}\in\{0,1\}^n} f(\bm{x})g(\bm{x}).
\]
Regarding the collections \(\mathcal{W}_{k}\) and \(\mathcal{W}_{G}\), we have the following important properties.

\begin{claim}\label{claim:orthogonal of distinct w_k}
    For any distinct \(k_{1}<k_{2}\in\{0,1,2,\ldots,s-1\}\), any polynomial \(f\in\mathcal{W}_{k_{1}}\) is orthogonal to any polynomial \(g\in\mathcal{W}_{k_{2}}\).
\end{claim}
\begin{poc}
Without loss of generality and by suitably relabeling, we can assume those polynomials of the form
\[
f=\prod_{j\in \bigsqcup_{h=1}^{k_{1}}A_{h}}x_j\prod_{j\in \bigsqcup_{h=k_{1}+1}^{s-1}A_{h}}(1-x_j),
\]
and 
\[
g=\prod_{j\in \bigsqcup_{h=1}^{k_{2}}B_{h}}x_j\prod_{j\in \bigsqcup_{h=k_{2}+1}^{s-1}B_{h}}(1-x_j)
\]
respectively, where \(A_{1},\ldots,A_{s-1},B_{1},\ldots,B_{s-1}\in\mathcal{F}.\) It suffices to show \(f\cdot g=0.\)
Since \(k_{1}<k_{2}\), we have \(k_{2}-k_{1}+s-1\ge s\). Moreover, since \(\mathcal{F}\) does not contain a matching of size \(s\), we have
\[
(A_{k_{1}+1}\sqcup\cdots \sqcup A_{s-1}) \cap (B_{1}\sqcup\cdots \sqcup B_{k_{2}})\neq \emptyset.
\]
Then there is some element \(a\in (A_{k_{1}+1}\sqcup\cdots \sqcup A_{s-1}) \cap (B_{1}\sqcup\cdots \sqcup B_{k_{2}})\), which yields that \(f\cdot g\) has a factor \(x_{a}\cdot (1-x_{a})\). Then \(f\cdot g=0\) since \(x_{a}\in\{0,1\}\), finishing the proof.
\end{poc}
Notice that by Claim~\ref{claim:orthogonal of distinct w_k}, we have
\begin{equation}\label{equ:Span}
    2^{n}\ge \sum_{k=0}^{s-1} \dim(\operatorname{span}(\mathcal{W}_{k})).
\end{equation}

\begin{claim}
    For distinct \(G_{1},G_{2}\in\mathcal{G}\), \(\mathcal{W}_{G_{1}}\cap \mathcal{W}_{G_{2}}=\emptyset.\)
\end{claim}
\begin{poc}
By definition of \(\mathcal{W}_{G}\), we can see for distinct \(G_{1},G_{2}\in\mathcal{G}\), we can arbitrarily decompose them as
\[
G_{1}=A_{1}\sqcup B_{1}\quad \textup{and}\quad G_{2}=A_{2}\sqcup B_{2}
\]
respectively, where \(A_{1},A_{2},B_{1},B_{2}\) are disjoint union of some sets in \(\mathcal{F}.\) Since \(G_{1},G_{2}\) are distinct, then either \(G_{1}\nsubseteq G_{2}\) or \(G_{2}\nsubseteq G_{1}\), without loss of generality, we can assume \(G_{1}\nsubseteq G_{2}\).
Suppose that
\[
w_{G_{1}}(\bm{x})=\prod_{j\in A_1}x_j\prod_{j\in B_1}(1-x_j)=\prod_{j\in A_2}x_j\prod_{j\in B_2}(1-x_j)=w_{G_{2}}(\bm{x}).
\]
By definition we can express \(w_{G_{1}}(\bm{x})\) and \(w_{G_{2}}(\bm{x})\) as 
\[
w_{G_{1}}(\bm{x})=\prod_{j\in A_1}x_j\prod_{j\in B_1}(1-x_j)=\sum_{A_1\subseteq S\subseteq G_1}(-1)^{|S\setminus A_1|}p_S(\bm{x})=\sum_{A_1\subseteq S\nsubseteq G_1}(-1)^{|S\setminus A_1|}p_S(\bm{x})+(-1)^{|B_1|}p_{G_1}(\bm{x}),
\]
and 
\[
w_{G_{2}}(\bm{x})=\prod_{j\in A_2}x_j\prod_{j\in B_2}(1-x_j)=\sum_{A_2\subseteq S\subseteq G_2}(-1)^{|S\setminus A_2|}p_S(\bm{x}),
\]
respectively. Since \(G_{1}\nsubseteq G_{2}\), we can see \((-1)^{|B_1|}p_{G_1}(\bm{x})\) does not appear in \(w_{G_{2}}(\bm{x})\) as a monomial. Then \(w_{G_{1}}(\bm{x})=w_{G_{2}}(\bm{x})\) implies that
\begin{equation}\label{eq:Sum0}
    \sum_{A_1\subseteq S\nsubseteq G_1}(-1)^{|S\setminus A_1|}p_S(\bm{x})+(-1)^{|B_1|}p_{G_1}(\bm{x})-\sum_{A_2\subseteq S\subseteq G_2}(-1)^{|S\setminus A_2|}p_S(\bm{x})=0.
\end{equation}
However, by Proposition~\ref{prop:linearly independent of p_s}, the polynomials in \(\{p_{S}\}_{S\subseteq [n]}\) are linearly independent, which is a contradiction to~\eqref{eq:Sum0}. This finishes the proof.
\end{poc}

\begin{claim}\label{claim:SudakovsBound}
    For each \(G\in\mathcal{G}\) and each \(k\in\{0,1,\ldots,s-1\}\), take arbitrary one polynomial \(w_{k,G}\in\mathcal{W}_{G}\cap \mathcal{W}_{k}\) respectively, then the polynomials in \(\{w_{k,G}\}_{G\in\mathcal{G},k\in \{0,1,\ldots,s-1\}}\) are linearly independent.
\end{claim}
\begin{poc}
By Claim~\ref{claim:orthogonal of distinct w_k}, it suffices to show that for fixed \(k\in \{0,1,2,\ldots,s-1\}\), the polynomials in \(\{w_{k,G}\}_{G\in\mathcal{G}}\) are linearly independent. Suppose that 
\begin{equation}\label{linearInd11}
    \sum_{G\in \mathcal{G}}\alpha_{G}\cdot w_{k,G}(\bm{x})=0
\end{equation}
and not all coefficients \(\alpha_{G}\) are zero. Notice that by Proposition~\ref{claim:RelationFG}, any set \(G\in\mathcal{G}\) can be expressed by \(G=X_{G}\sqcup Y_{G}\), where \(X_{G}\) and \(Y_{G}\) are disjoint union of some sets in \(\mathcal{F}\). Let \(H\) be a maximal set in \(\mathcal{G}\) such that \(\alpha_{H}\neq 0\), that is, there is no \(H'\in\mathcal{G}\) containing \(H\) such that \(\alpha_{H'}\neq 0\). Then we can see
\[
\sum_{G\in \mathcal{G}}\alpha_{G}\cdot w_{k,G}(\bm{x})=\alpha_{H}\cdot w_{k,H}(\bm{x})+\sum_{G\in \mathcal{G},H\nsubseteq G}\alpha_{G}\cdot w_{k,G}(\bm{x}).
\]
Therefore, by~\eqref{linearInd11} we have
\[
\alpha_{H}\cdot \sum_{X_{H}\subseteq S\subseteq H}(-1)^{|S\setminus X_{H}|}p_S(\bm{x}) + \sum_{G\in \mathcal{G},H\nsubseteq G}\bigg(\alpha_{G}\cdot\sum\limits_{X_{G}\subseteq S\subseteq G}(-1)^{|S\setminus X_{G}|}p_{S}(\bm{x}) \bigg)=0,
\]
which implies that
\[
\alpha_{H}\cdot (-1)^{|Y_{H}|}p_{H}(\bm{x})+
\alpha_{H}\cdot \sum_{X_{H}\subseteq S\nsubseteq H}(-1)^{|S\setminus X_{H}|}p_S(\bm{x}) + \sum_{G\in \mathcal{G},H\nsubseteq G}\bigg(\alpha_{G}\cdot\sum\limits_{X_{G}\subseteq S\subseteq G}(-1)^{|S\setminus X_{G}|}p_{S}(\bm{x}) \bigg)=0,
\]
a contradiction to Proposition~\ref{prop:linearly independent of p_s} since \(\alpha_{H}\cdot (-1)^{|Y_{H}|}\neq 0\).
\end{poc}
Observe that Claim~\ref{claim:SudakovsBound} already implies \(s\cdot|\mathcal{G}|\le 2^{n}\), which in turn yields the lower bound in~\cref{thm:2018BLMS}. We then further refine this method.

If for any \(G\in\mathcal{G}\), \(|G|> n-g(n,s)\), then \(|\mathcal{G}|\le\sum\limits_{i=n-g(n,s)}^{n}\binom{n}{i} \), we are done. Otherwise, we can take a set \(G_{0}\in\mathcal{G}\) with \(|G_{0}|\le n-g(n,s)\).  Fix this \(G_{0}\), let
\[
\mathcal{G}_{0}:=\{G\in\mathcal{G}:G_{0}\subseteq G\}.
\]
By Corollary~\ref{Cor:IncrGG}, we can see
\[
|\mathcal{G}_{0}|\ge 2^{g(n,s)}.
\]
By Proposition~\ref{claim:RelationFG}, we can take and fix some decomposition of \(G_{0}\) as 
\[
G_{0}=B_{1}\sqcup B_{2}\sqcup\cdots\sqcup B_{s-1},
\]
where \(B_{1},\ldots,B_{s-1}\in\mathcal{F}\). Also by Proposition~\ref{claim:Increasing}, for any element \(G\in \mathcal{G}_{0}\), we can select a decomposition 
\[
G=C_{1}^{G}\sqcup C_{2}^{G}\sqcup\cdots\sqcup C_{s-1}^{G},
\]
such that \(C_{1}^{G},\ldots,C_{s-1}^{G}\in\mathcal{F}\) and
\begin{equation}\label{containment:BjCj}
    B_{j}\subseteq C_{j}^{G}
\end{equation}
for any \(j\in [s-1].\)

We now state the key lemma.
\begin{lemma}\label{prop:Final}
    For any \(k\in\{0,1,2,\ldots,s-1\}\), \(\dim(\operatorname{span}(\mathcal{W}_{k}))\ge |\mathcal{G}|+\left(\binom{s-1}{k}-1\right)\cdot 2^{g(n,s)}.\)
\end{lemma}
\begin{proof}[Proof of Lemma~\ref{prop:Final}]
    Let \(k\in\{0,1,2,\ldots,s-1\}\). By our selection, for any element \(G\in \mathcal{G}_{0}\), there is a decomposition 
\[
G=C_{1}^{G}\sqcup C_{2}^{G}\sqcup\cdots\sqcup C_{s-1}^{G},
\]
such that \(C_{1}^{G},\ldots,C_{s-1}^{G}\in\mathcal{F}\) and \(B_{j}\subseteq C_{j}^{G}\) for any \(j\in [s-1].\) Moreover, notice that there are \(\binom{s-1}{k}\) ways to partition \([s-1]\) into \([s-1]=S\sqcup T\) such that \(|S|=k\). Then for each fixed partition \([s-1]=S\sqcup T\) and each \(G\in\mathcal{G}_{0}\), we define a polynomial
\[
p_{G,k,S}(\bm{x})=\prod\limits_{j\in \bigsqcup_{h\in S}C_{h}^{G}}x_{j}\cdot \prod\limits_{j\in \bigsqcup_{h\in T}C_{h}^{G}}(1-x_{j}).
\]
On the other hand, for any \(H\in\mathcal{G}\setminus\mathcal{G}_{0}\), by Proposition~\ref{claim:RelationFG}, we can take and fix some decomposition
\[
H=D_{1}^{H}\sqcup D_{2}^{H}\sqcup\cdots\sqcup D_{s-1}^{H},
\]
where \(D_{1}^{H},\ldots,D_{s-1}^{H}\in\mathcal{F}.\) Since the matching number of \(\mathcal{F}\) is \(s-1\), then for any \(t\in [s-1]\), there exists some \(r_{t}\in [s-1]\) such that \(B_{t}\cap D_{r_{t}}^{H}\neq\emptyset.\) Therefore, for \(B_{1},B_{2},\ldots,B_{k}\), we can select at most \(k\) many distinct sets in \(\{D_{1}^{H},\ldots,D_{s-1}^{H}\}\) such that their union intersects each \(B_{h}\) for \(h\in [k]\). Without loss of generality and by suitably relabeling, we can assume that
\begin{equation}\label{RelationBDBD}
    \bigg(\bigsqcup_{j=1}^{k}D_{j}^{H}\bigg)\cap B_{h}\neq\emptyset
\end{equation}
for any \(h\in [k].\) Based on those rules of decomposition of any element in \(\mathcal{G}\setminus\mathcal{G}_{0}\), we then define a polynomial for \(H\in\mathcal{G}\setminus\mathcal{G}_{0}\) as
\[
q_{H,k}(\bm{x}):=\prod\limits_{j\in\bigsqcup_{h=1}^{k}D_{h}^{H}}x_{j}\cdot \prod\limits_{j\in\bigsqcup_{h=k+1}^{s-1}D_{h}^{H}}(1-x_{j}).
\]
It then suffices to show the linear independence of the above polynomials.
\begin{prop}\label{claim:FinalLinearInd}
   The polynomials in \(\{p_{G,k,S}:G\in \mathcal{G}_{0},S\subseteq [s-1],|S|=k\}\sqcup \{q_{H,k}: H\in \mathcal{G}\setminus\mathcal{G}_{0}\}\) defined above are linearly independent. 
\end{prop}
\begin{proof}[Proof of Proposition~\ref{claim:FinalLinearInd}]
We first show two auxiliary claims.
\begin{claim}\label{claim:PPPP}
   For any \(k\in [s-1]\), any distinct \(S\) and \(S'\) with \(|S|=|S'|=k\), and any \(G,G'\in\mathcal{G}_{0}\), \(p_{G,k,S}\) is orthogonal to \(p_{G',k,S'}\).
\end{claim}
\begin{poc}
    {Since \(S\neq S'\) and \(|S|=|S'|=k\), there exists \(h\in S\setminus S'\).
By definition, \(p_{G,k,S}(\bm{x})\) contains the factor \(\prod_{j\in C_h^{G}} x_j\), while \(p_{G',k,S'}(\bm{x})\) contains the factor \(\prod_{j\in C_h^{G'}} (1-x_j)\).
By~\eqref{containment:BjCj}, we have \(B_h\subseteq C_h^{G}\cap C_h^{G'}\), then the pointwise product \(p_{G,k,S}(\bm{x})\cdot p_{G',k,S'}(\bm{x})\) contains
\[
\prod_{j\in B_h} x_j(1-x_j),
\]
which vanishes for every \(\bm{x}\in\{0,1\}^n\).
This finishes the proof.}
\end{poc}
\begin{claim}\label{Claim:PQPQ}
    For any \(k\in [s-1]\), any \(G\in\mathcal{G}_{0}\) and any $H\in \mathcal{G}\setminus \mathcal{G}_{0}$, 
     $q_{H,k}$ is orthogonal to $p_{G,k,S}$ if $S\neq [k]$.
\end{claim}
\begin{poc}
    { Assume \(S\neq [k]\) and choose \(u\in [k]\setminus S\).
By definition, \(p_{G,k,S}(\bm{x})\) contains the factor \(\prod_{j\in C_u^{G}} (1-x_j)\),
while \(q_{H,k}(\bm{x})\) contains the factor \(\prod_{j\in \bigsqcup_{i=1}^{k} D_i^{H}} x_j\).
By \eqref{containment:BjCj} we have \(B_u\subseteq C_u^{G}\), and by \eqref{RelationBDBD} we have
\[
\bigg(\bigsqcup_{i=1}^{k} D_i^{H}\bigg)\cap B_u \neq \emptyset .
\]
Hence there exists \(t\in C_u^{G}\cap \big(\bigsqcup_{i=1}^{k} D_i^{H}\big)\).
Consequently, the pointwise product \(p_{G,k,S}(\bm{x})\, q_{H,k}(\bm{x})\) contains the factor
\(x_t(1-x_t)\), which vanishes for every \(\bm{x}\in\{0,1\}^n\).
Thus \(p_{G,k,S}\cdot q_{H,k}= 0\) on \(\{0,1\}^n\). This finishes the proof.}
\end{poc}

Now suppose that
\begin{equation}\label{eq:Assumption}
    \sum_{G\in \mathcal{G}_{0},|S|=k}\alpha_{G,S}\cdot p_{G,k,S}+\sum_{H\in \mathcal{G}\setminus\mathcal{G}_{0}}\beta_{H}\cdot q_{H,k}=0.
\end{equation}
By Claim~\ref{claim:orthogonal of distinct w_k}, Claim~\ref{claim:PPPP} and Claim~\ref{Claim:PQPQ}, we can see
\[\sum_{G\in \mathcal{G}_{0}}\alpha_{G,[k]}\cdot p_{G,k,[k]}+\sum_{H\in \mathcal{G}\setminus\mathcal{G}_{0}}\beta_{H}\cdot q_{H,k}\] is orthogonal to each $p_{G,k,S}$, where $S\neq [k]$. Therefore, we have
\[
\begin{aligned}
0&=\bigg(\sum_{G\in \mathcal{G}_{0}}\alpha_{G,[k]}\cdot p_{G,k,[k]}+\sum_{H\in \mathcal{G}\setminus\mathcal{G}_{0}}\beta_{H}\cdot q_{H,k}\bigg)\cdot\bigg(\sum_{G\in \mathcal{G}_{0},|S|=k}\alpha_{G,S}\cdot p_{G,k,S}+\sum_{H\in \mathcal{G}\setminus\mathcal{G}_{0}}\beta_{H}\cdot q_{H,k}\bigg)\\
&=\bigg(\sum_{G\in \mathcal{G}_{0}}\alpha_{G,[k]}\cdot p_{G,k,[k]}+\sum_{H\in \mathcal{G}\setminus\mathcal{G}_{0}}\beta_{H}\cdot q_{H,k}\bigg)^2.
\end{aligned}
\]
   It then follows that
     \begin{equation}\label{eq1}
         \sum_{G\in \mathcal{G}_{0}}\alpha_{G,[k]}\cdot p_{G,k,[k]}+\sum_{H\in \mathcal{G}\setminus\mathcal{G}_{0}}\beta_{H}\cdot q_{H,k}=0.
     \end{equation}
     Moreover, by~\eqref{eq:Assumption} and~\eqref{eq1}, we have
      \begin{equation}\label{eq2}
         \sum_{G\in \mathcal{G}_{0},S\in \binom{[s-1]}{k}\setminus {[k]}}\alpha_{G,S}\cdot p_{G,k,S}=0.
     \end{equation}
     For any two distinct $S_1,S_2\in \binom{[s-1]}{k}\setminus \{[k]\}$, by Claim~\ref{claim:PPPP}, we have 
     $$\bigg(\sum_{G\in \mathcal{G}_{0}}\alpha_{G,S_1}\cdot p_{G,k,S_1}\bigg)\cdot \bigg(\sum_{G\in \mathcal{G}_{0}}\alpha_{G,S_2}\cdot p_{G,k,S_2}\bigg)=0.$$
     Moreover, by~\eqref{eq2}, we have
     $$0=\bigg(\sum_{G\in \mathcal{G}_{0}}\alpha_{G,S_1}\cdot p_{G,k,S_1}\bigg)\cdot\bigg(\sum_{G\in \mathcal{G}_{0},S\in \binom{[s-1]}{k}\setminus {[k]}}\alpha_{G,S}\cdot p_{G,k,S}\bigg)=\bigg(\sum_{G\in \mathcal{G}_{0}}\alpha_{G,S_1}\cdot p_{G,k,S_1}\bigg)^2.$$
     Hence, for each $S\in \binom{[s-1]}{k}\setminus\{[k]\}$, we have
     \begin{equation}\label{eq3}
         \sum_{G\in \mathcal{G}_{0}}\alpha_{G,S}\cdot p_{G,k,S}=0.
     \end{equation}
      Note that for each $G\in \mathcal{G}$ there is at most one element of $\mathcal{W}_G$ appearing in~\eqref{eq1}, by the linear independence given by Claim~\ref{claim:SudakovsBound}, we conclude that $\alpha_{G,[k]}=0,\beta_{H}=0$ for all $G\in \mathcal{G}_0$ and $H\in \mathcal{G}\setminus\mathcal{G}_{0}$.
     Similarly, since for each $G\in \mathcal{G}_{0}$ there is at most one element of $\mathcal{W}_G$ appearing in~\eqref{eq3}, we have $\alpha_{G,S}=0$ for any $G\in \mathcal{G}_{0}$ and $S\neq [k]$. Then we can see if assumption~\eqref{eq:Assumption} holds, then all of the coefficients are zero. This finishes the proof.

\end{proof}
Then by Proposition~\ref{claim:FinalLinearInd}, we have
\[
\binom{s-1}{k}\cdot|\mathcal{G}_{0}|+(|\mathcal{G}|-|\mathcal{G}_{0}|)\le \dim(\operatorname{span}(\mathcal{W}_{k})).
\]
This finishes the proof.
\end{proof}

By Lemma~\ref{prop:Final} and~\eqref{equ:Span}, we can see 
\[
2^{n}\ge \sum\limits_{k=0}^{s-1}\dim(\operatorname{span}(\mathcal{W}_{k}))\ge s\cdot |\mathcal{G}| + \sum\limits_{k=0}^{s-1}\bigg(\binom{s-1}{k}-1\bigg)\cdot |\mathcal{G}_{0}|. 
\]
Then the desired bound in~\cref{thm:main} follows.

\section{Concluding remarks}
In this paper, we proved that every \(s\)-saturated family \(\mathcal{F}\subseteq 2^{[n]}\) satisfies
\[
|\mathcal{F}|
\;\ge\;
\Bigl(1-\frac{1}{s}+\Omega_{s}\Bigl(\frac{\log n}{n}\Bigr)\Bigr)\cdot 2^{n}.
\]
Our proof proceeds by combining the structural identity
\(
\overline{\mathcal{F}^{c}}=\mathcal{F}^{\Box (s-1)}
\)
with the celebrated Kahn--Kalai--Linial theorem, and thus reduces the problem to finding a coordinate
\(x\in [n]\) for which the section
\(
\mathcal{F}^{\Box(s-1)}(\overline{x})
\)
is noticeably smaller than half of \(\mathcal{F}^{\Box(s-1)}\).

A natural hope would be to prove the existence of a constant \(\varepsilon=\varepsilon(s)>0\) such that
\[
\min_{x\in [n]}
\bigl|\mathcal{F}^{\Box(s-1)}(\overline{x})\bigr|
\;\le\;
\Bigl(\frac12-\varepsilon\Bigr)\cdot
\bigl|\mathcal{F}^{\Box(s-1)}\bigr|
\]
holds for every \(s\)-saturated family \(\mathcal{F}\). Combined with the inequality established in the proof of \cref{thm:ViaKKL}, such a bound would immediately imply
\[
|\mathcal{F}|
\;\ge\;
\Bigl(1-\frac{1}{s}+\Omega_s(1)\Bigr)\cdot 2^{n},
\]
which would be a much stronger improvement over \cref{thm:2018BLMS}.

However, the following proposition shows that this approach cannot work in general. In particular, there is no uniform constant \(\varepsilon>0\) for which the above section bound holds for all \(s\)-saturated families.

\begin{prop}\label{prop:section-ratio-sharp}
For every fixed integer \(s\ge 2\), there exists a sequence of \(s\)-saturated families
\(
\mathcal{F}_{n}\subseteq 2^{[n]}
\)
such that
\[
\lim_{n\to\infty}
\frac{\min_{x\in [n]} \bigl|\mathcal{F}_{n}^{\Box(s-1)}(\overline{x})\bigr|}
     {\bigl|\mathcal{F}_{n}^{\Box(s-1)}\bigr|}
=
\frac{1}{2}.
\]
\end{prop}
\begin{proof}[Proof of Proposition~\ref{prop:section-ratio-sharp}]
    Fix \(s\ge 2\). For sufficiently large \(n\), let \(m=m(n)\) be the largest odd integer such that
\(
(s-1)m\le n.
\)
Write
\(
m=2r+1.
\)
Since \(s\) is fixed, we have \(m(n)\to\infty\) and hence \(r\to\infty\) as \(n\to\infty\).

Partition \([n]\) as
\[
[n]=I_{1}\sqcup I_{2}\sqcup \cdots \sqcup I_{s-1}\sqcup J,
\]
where
\(
|I_{1}|=\cdots=|I_{s-1}|=m.
\) For each \(i\in [s-1]\), define
\[
\mathcal{H}_{i}
\coloneqq
\{A\subseteq I_{i}: |A|\ge r+1\}.
\]
Now define
\[
\mathcal{F}_{n}
\coloneqq
\bigl\{
A\subseteq [n]:
A\cap I_{i}\in \mathcal{H}_{i}\text{ for some }i\in [s-1]
\bigr\}.
\]
\begin{claim}\label{claim:FnSaturated}
    \(\mathcal{F}_{n}\) is \(s\)-saturated.
\end{claim}
\begin{poc}
  To see that \(\mathcal{F}_{n}\) contains no matching of size \(s\), suppose for contradiction that
\[
A_{1},A_{2},\dots,A_{s}\in \mathcal{F}_{n}
\]
are pairwise disjoint. For each \(j\in [s]\), choose \(i(j)\in [s-1]\) such that
\(
A_{j}\cap I_{i(j)}\in \mathcal{H}_{i(j)}.
\) Since there are only \(s-1\) possible values of \(i(j)\), two of these sets, say \(A_{p}\) and \(A_{q}\), satisfy
\(
i(p)=i(q)=t
\)
for some \(t\in [s-1]\). But then
\[
A_{p}\cap I_{t},\ A_{q}\cap I_{t}\in \mathcal{H}_{t},
\]
and each of these sets has size at least \(r+1\). Since \(|I_{t}|=2r+1\), any two subsets of \(I_t\) of size at least \(r+1\) must intersect. Hence
\[
(A_{p}\cap I_{t})\cap (A_{q}\cap I_{t})\neq\emptyset,
\]
which implies \(A_{p}\cap A_{q}\neq\emptyset\), a contradiction. 

Next we prove maximality. Let \(B\subseteq [n]\) be such that \(B\notin \mathcal{F}_{n}\). Then for every \(i\in [s-1]\),
\(
|B\cap I_{i}|\le r.
\)
Hence, for each \(i\in [s-1]\), the set
\(
C_{i}\coloneqq I_{i}\setminus B
\)
satisfies
\[
|C_{i}|=|I_i|-|B\cap I_i|\ge (2r+1)-r=r+1,
\]
so \(C_i\in \mathcal{H}_i\). Therefore \(C_i\in \mathcal{F}_n\) for every \(i\in [s-1]\). Since the blocks \(I_1,\dots,I_{s-1}\) are pairwise disjoint, the sets \(C_1,\dots,C_{s-1}\) are pairwise disjoint, and each is also disjoint from \(B\). Thus
\(
B,C_1,\dots,C_{s-1}
\)
form a matching of size \(s\). This finishes the proof.
\end{poc}
We now determine \(\mathcal{F}_{n}^{\Box(s-1)}\). By Proposition~\ref{claim:RelationFG},
\[
\mathcal{F}_{n}^{\Box(s-1)}=\overline{\mathcal{F}_{n}^{\,c}}.
\]
Therefore, for \(A\subseteq [n]\),
\(
A\in \mathcal{F}_{n}^{\Box(s-1)}
\) if and only if
\(
[n]\setminus A\notin \mathcal{F}_{n}.
\)
By the definition of \(\mathcal{F}_n\), this is equivalent to saying that for every \(i\in [s-1]\),
\(
([n]\setminus A)\cap I_i \notin \mathcal H_i.
\)
Since \(\mathcal H_i\) consists precisely of the subsets of \(I_i\) of size at least \(r+1\), the above condition is equivalent to
\(
|([n]\setminus A)\cap I_i|\le r
\)
for all \(i\in [s-1],\) which further yields that
\(
|A\cap I_i|\ge r+1
\)
for all \(i\in [s-1].\)
Hence
\[
\mathcal{F}_{n}^{\Box(s-1)}
=
\bigl\{
A\subseteq [n]:
|A\cap I_i|\ge r+1 \text{ for every }i\in [s-1]
\bigr\}.
\]
We next compute the relevant section ratios. First note that for each \(i\in [s-1]\),
\[
|\mathcal H_i|
=
\sum_{u=r+1}^{2r+1}\binom{2r+1}{u}
=
2^{2r},
\]
by symmetry of the binomial coefficients. Now fix \(x\in J\). Since membership in \(\mathcal{F}_{n}^{\Box(s-1)}\) imposes no condition on the coordinates in \(J\), exactly half of the sets in \(\mathcal{F}_{n}^{\Box(s-1)}\) avoid \(x\). Thus
\[
\frac{\bigl|\mathcal{F}_{n}^{\Box(s-1)}(\overline{x})\bigr|}
     {\bigl|\mathcal{F}_{n}^{\Box(s-1)}\bigr|}
=
\frac12.
\]

Next, fix \(x\in I_{j}\) for some \(j\in [s-1]\). Then a set in \(\mathcal{F}_{n}^{\Box(s-1)}\) avoiding \(x\) must still contain at least \(r+1\) points from \(I_j\setminus\{x\}\), which has size \(2r\). Therefore
\[
\frac{\bigl|\mathcal{F}_{n}^{\Box(s-1)}(\overline{x})\bigr|}
     {\bigl|\mathcal{F}_{n}^{\Box(s-1)}\bigr|}
=
\frac{\sum_{u=r+1}^{2r}\binom{2r}{u}}
     {\sum_{u=r+1}^{2r+1}\binom{2r+1}{u}}
=
\frac{\sum_{u=r+1}^{2r}\binom{2r}{u}}{2^{2r}}.
\]
Again by symmetry,
\(
\sum_{u=r+1}^{2r}\binom{2r}{u}
=
2^{2r-1}-\frac12\binom{2r}{r}.
\)
Hence
\[
\frac{\bigl|\mathcal{F}_{n}^{\Box(s-1)}(\overline{x})\bigr|}
     {\bigl|\mathcal{F}_{n}^{\Box(s-1)}\bigr|}
=
\frac12-\frac{1}{2^{2r+1}}\binom{2r}{r}.
\]
Then we obtain that
\[
\frac{\min_{x\in [n]} \bigl|\mathcal{F}_{n}^{\Box(s-1)}(\overline{x})\bigr|}
     {\bigl|\mathcal{F}_{n}^{\Box(s-1)}\bigr|}
=
\frac12-\frac{1}{2^{2r+1}}\binom{2r}{r}.
\]
Notice that by Stirling's formula,
\(
\frac{1}{2^{2r+1}}\binom{2r}{r}\to 0
\)
when \(r\rightarrow\infty.\)
It follows that
\[
\lim_{n\to\infty}
\frac{\min_{x\in [n]} \bigl|\mathcal{F}_{n}^{\Box(s-1)}(\overline{x})\bigr|}
     {\bigl|\mathcal{F}_{n}^{\Box(s-1)}\bigr|}
=
\frac12.
\]
This finishes the proof.
\end{proof}
Proposition~\ref{prop:section-ratio-sharp} shows that the KKL-based argument developed in this paper is close to the natural limit of what one can hope to obtain from a universal Boolean-function inequality alone. In particular, any proof of a bound of the form
\[
|\mathcal{F}|
\;\ge\;
\Bigl(1-\frac{1}{s}+\Omega_s(1)\Bigr)\cdot 2^{n}
\]
will necessarily require additional structural input that goes beyond the existence of a single coordinate with an atypically small section of \(\mathcal{F}^{\Box(s-1)}\). We therefore believe that a genuinely substantial improvement on the conjecture of Erd\H{o}s and Kleitman will require new ideas that exploit finer features of \(s\)-saturated families.

\bibliographystyle{abbrv}
\bibliography{saturated}
\end{document}